\documentclass[12pt]{amsart}
\usepackage{graphicx}
\usepackage{amsmath, enumerate}
\usepackage{amsfonts, amssymb, amsthm, xcolor, stmaryrd}
\usepackage{tikz, tikz-cd, caption, tabu, longtable, mathrsfs, xtab, centernot, mathtools}
\usepackage{leftidx}
 \usepackage{dsfont, cite, todonotes}
\usepackage[toc,page]{appendix}
\usepackage[pagebackref=true]{hyperref}
\AtBeginDocument{%
   \def\MR#1{}
}

\numberwithin{equation}{section}
\theoremstyle{plain}

\newtheorem{thm}{Theorem}[section]
\newtheorem{lemma}[thm]{Lemma}
\newtheorem{prop}[thm]{Proposition}
\newtheorem{cor}[thm]{Corollary}

\theoremstyle{definition}

\theoremstyle{remark}
\newtheorem{rmk}[thm]{Remark}

\newcommand{\Hb}{\mathbb{H}}
\newcommand{\Zb}{\mathbb{Z}}
\newcommand{\SL}{{\mathrm{SL}}}
\newcommand{\PSL}{{\mathrm{PSL}}}
\newcommand{\Cb}{\mathbb{C}}
\newcommand{\Qb}{\mathbb{Q}}
\newcommand{\pf}{\mathfrak{p}}

\newcommand{\Nb}{\mathbb{N}}
\newcommand{\Rb}{\mathbb{R}}
\newcommand{\Nm}{{\mathrm{Nm}}}
\newcommand{\Fc}{{\mathcal{F}}}
\newcommand{\tth}{\textsuperscript{th }}
\newcommand{\lp}{\left (}
\newcommand{\rp}{\right )}
\newcommand{\Oc}{\mathcal{O}}

\newcommand{\Cl}{{\mathrm{Cl}}}
\newcommand{\Gal}{{\mathrm{Gal}}}
\newcommand{\af}{\mathfrak{a}}

\newcommand{\ef}{\mathfrak{e}}

\newcommand{\Cc}{\mathcal{C}}
\newcommand{\Tr}{\mathrm{Tr}}

\newcommand{\ebf}{{\mathbf{e}}}
\newcommand{\Ab}{{\mathbb{A}}}

\newcommand{\zm}{\mathrm{z}}

\newcommand{\pmat}[4]{\begin{pmatrix}
                 #1 & #2\\
                 #3 & #4
\end{pmatrix}}
\begin{document}
\title{Singular Units and Isogenies Between CM Elliptic Curves}
\author[Yingkun Li]{Yingkun Li}
\address{Fachbereich Mathematik,
Technische Universit\"at Darmstadt, Schlossgartenstrasse 7, D--64289
Darmstadt, Germany}
\email{li@mathematik.tu-darmstadt.de}
\thanks{The author is partially supported by the LOEWE research unit USAG.}

\date{\today}
\maketitle
\begin{abstract}
In this note, we will apply the results of Gross-Zagier, Gross-Kohnen-Zagier and their generalizations to give a short proof that the differences of singular moduli are not units. As a consequence, we obtain a result on isogenies between reductions of CM elliptic curves. 
\end{abstract}

\section{Introduction}
Let $\Hb$ be the complex upper half-plane, which is acted on discretely by the group $\Gamma := \PSL_2(\Zb)$. The modular curve $Y(1) := \Gamma \backslash \Hb$ is the coarse moduli space of isomorphism classes of elliptic curves over $\Cb$. 
The Klein-$j$ invariant provides the uniformization 
$$
j: Y(1) \to \Cb.
$$
Let $\zm \in Y(1)$ be a CM point of discriminant $d < 0$, i.e.\ it corresponds to an elliptic curve $E_\zm$ with complex multiplication by the order $\Oc_d := \Zb + \tfrac{d + \sqrt{d}}{2}\Zb$ in the imaginary quadratic field $K = \Qb(\sqrt{d})$. 
Then $j(\zm)$ is called a \textit{singular modulus}.
The classical theory of complex multiplication tells us that $j(\zm)$ is an algebraic integer generating the ring class field $H_d$ of $K$, which corresponds to the order $\Oc_d$ via class field theory.
The factorization of the difference of singular moduli was the subject of the seminal work by Gross and Zagier \cite{GZ85} and has interesting implications for the CM elliptic curve $E_\zm$. For example, if $j(\zm)$ is divisible by a prime $\pf$ in $H_d$, then the reduction of $E_\zm$ modulo $\pf$ is isomorphic to the reduction of the CM elliptic curve
$$
E: y^2 = x^3 - 1,
$$
whose corresponding $j$-invariant is zero.

More generally for $m \ge 1$, let $\varphi_m(X, Y) \in \Zb[X, Y]$ be the modular polynomial defined by 
\begin{equation}
  \label{eq:modpol}
  \varphi_m(j(z_1), j(z_2)) := \prod_{\begin{subarray}{c} \gamma \in \Gamma \backslash \Gamma_m \\ \end{subarray}} (j(z_1) - j(\gamma z_2)),
\end{equation}
where $\Gamma_m$ consists of matrices in $\mathrm{PGL}_2(\Zb)$ with determinant equals to $\pm m$, and is acted on by $\Gamma$ via multiplication on the left.
For example, $\varphi_1(X, Y)$ is simply $X - Y$.
A prime that divides $\varphi_m(j(\zm_1), j(\zm_2))$ for two CM points $\zm_1, \zm_2$ then gives us a finite field, over which the reductions of $E_{\zm_1}$ and $E_{\zm_2}$ are $m$-isogenous.

In this note, we will apply the results by Gross-Zagier \cite{GZ85}, Gross-Kohnen-Zagier \cite{GKZ87} and their generalizations by Schofer \cite{Schofer}, Bruinier-Kudla-Yang \cite{BKY12}, Bruinier-Ehlen-Yang \cite{BEY} to prove the following result.
\begin{thm}
\label{thm:main}
Let $m \in \Nb$ and $\zm_1, \zm_2 \in Y(1)$ be CM points of discriminants $d_1, d_2 < 0$, and $H$ the composite of the ring class fields $H_{d_i}$. Then the norm of the algebraic integer $\varphi_m(j(\zm_1), j(\zm_2)) \in H$, if non-zero, satisfies the lower bound 
\begin{equation}
  \label{eq:lowerbd}
\log  |  \Nm_{H/\Qb} \varphi_m(j(\zm_1), j(\zm_2))| \ge
 |Z(W) \cap T_{m, \epsilon}|
\cdot  Q_2(\cosh^{} (\sqrt{2} \epsilon))
\end{equation}
 for any $\epsilon > 0$.
Here $Q_{s-1}$ is the Legendre function of the second kind, $Z(W) \subset \Fc^2 \subset \Hb^2$ the set of Galois conjugates of $(\zm_1, \zm_2)$ (see \eqref{eq:bigCM1} and \eqref{eq:ZW2}) with $\Fc \subset \Hb$ a fundamental domain of $Y(1)$, and $T_{m, \epsilon} \subset \Hb^2$ the $\epsilon$-neighborhood (with respect to the Riemannian metric on $\Hb^2$) of the graph of the $m$\tth Hecke correspondence.
\end{thm}

\begin{rmk}
  From the definition of $Q_{s-1}$ in \eqref{eq:Qs}, it is clear that $Q_{s-1}(t)$ is positive and monotonic for $s, t > 1$. 
In fact, $Q_{s-1}(t) \gg_s t^{-s} \log \frac{t + 1}{t-1}$.
We can therefore make $\epsilon$ large enough such that $Z(W) \cap {T}_{m, \epsilon} \neq \emptyset$.
\end{rmk}

In \cite{BHK18}, it was shown that $j(\zm)$ is not a unit for any CM point $\zm \in Y(1)$. This was an improvement over an earlier result for all (ineffectively) large discriminants in \cite{Habegger15}.
By specializing Theorem \ref{thm:main} to $m = 1$ and $d_1 = -3$, we recover the main result of \cite{BHK18}. By allowing both CM points to vary, we actually have the following more general result.

\begin{cor}
  \label{cor:BHK}
For any CM points $\zm_1, \zm_2 \in Y(1)$ and integer $m \ge 1$, the algebraic integer $\varphi_m(j(\zm_1) , j(\zm_2))$ is never a unit. 
\end{cor}

\begin{rmk}
  \label{rmk:1}
When one of the discriminants is fixed, the proof in \cite{BHK18} can be adapted to prove the result above. However, this involves eliminating finitely many cases by computer calculation, and it is not clear if the same strategy works with both discriminants varying.
\end{rmk}

\begin{rmk}
  \label{rmk:2}
By definition, $\varphi_m(j(z_1, z_2)) = 0$ if and only if $(z_1, z_2)$ lies on the graph of the $m$\tth Hecke correspondence.
In particular, the corollary above implies that $\varphi_m(j(\zm), j(\zm))$ is never a unit for any CM point $\zm \in Y(1)$. 
Note that this value is not zero as long as $m$ is not a perfect square.
\end{rmk}
The results in \cite{Habegger15} and \cite{BHK18} originated from a question of Masser, which was motivated by effective results of Andr\'{e}-Oort type. As a generalization of Corollary 1.2 in \cite{BHK18}, we can deduce the following result from Theorem \ref{thm:main}.
\begin{cor}
  \label{cor:effective}
Let $P$ be a polynomial in unknowns $X_1,\dots, X_n$ with coefficients that are algebraic integers in $\Cb$. If $P$ is divisible by the $m$\tth modular polynomial $\varphi_m(X_i, X_j)$ for some $m \ge 1$ and $1 \le i \le j \le n$, then the subvariety in $\Cb^n$ defined by the equation $P(X_1, \dots, X_n) = 1$ contains no special points. 
\end{cor}
By the discussion concerning isogenies between elliptic curves, Theorem \ref{thm:main} also implies the following result.

\begin{cor}
  \label{cor:isogeny}
For $i = 1, 2$, let $E_i$ be an elliptic curves with CM by the order $\Oc_{d_i}$.
For any $m \in \Nb$, there exists a prime $\pf$ of $H_{d_1}H_{d_2}$ such that the reductions of $E_1$ and $E_2$ modulo $\pf$ are $m$-isogenous. 
\end{cor}

The idea of the proof of Theorem \ref{thm:main} is rather simple. 
In a nutshell, the result of Gross-Zagier expressed the left hand side of \eqref{eq:lowerbd} as a finite sum of non-negative quantities. 
Then a special case of the result of Gross-Kohnen-Zagier expressed the special value of higher Green's function as a different linear combination of these non-negative quantities. From its definition as a Poincar\'{e} series, the higher Green's function clearly never vanishes. This then tells us that these non-negative quantities are not all zero. One can even obtain a bound as in \eqref{eq:lowerbd}. 

In terms of Arakelov theory, the factorization of Gross-Zagier comes from explicitly calculating the archimedean and non-archimedean contributions to the self-intersection of Heegner points on the modular curve $Y(1)$, which add up to zero. 
The archimedean part is the negative of the norm of the difference of singular moduli, and the non-archimedean part gives the factorization. 
In the higher weight case, one would be calculating the self-intersection of Heegner cycles on Kuga-Sato varieties \cite{Zhang97}. 
This is still zero in some cases, and the local contributions are closely related to the case of modular curve. The advantage though is that the archimedean contribution in the higher weight setting is visibly non-zero from definition. We then use the non-archimedean contributions as a bridge to pass this information to the modular curve case.

Despite its simplicity, this novel idea is rather robust and most of the tools used are available in more general settings. In particular, we hope to apply this idea to study the genus 2 case and deduce analogues of results in \cite{HP17}.

In an earlier version, there were some conditions on the discriminants $d_1, d_2$, which are inherent in the results of Gross-Zagier and Gross-Kohnen-Zagier. These are now removed by the more general results by Schofer \cite{Schofer}, Bruinier-Kudla-Yang \cite{BKY12} and Bruinier-Ehlen-Yang \cite{BEY}. 
To apply these more general results, one needs to identify certain toric orbits of CM points with suitable Galois orbits. 
In the case of singular moduli, this works out nicely when $d_1, d_2$ are coprime and fundamental (see e.g.~ section 3.2 of \cite{YY19}). Otherwise, one can use the crucial fact that singular moduli generates ring class fields to still make suitable identifications.
This is contained in Proposition \ref{prop:classgp}, which is rather interesting and useful by itself, as one can use it to remove the conditions on the discriminants in the result of \cite{GZ85} and prove Conjecture 1.7 in \cite{LV15} (see section 4 of \cite{YY19} for the general strategy).

Another essential ingredient is the non-negativity of Fourier coefficients of certain incoherent Eisenstein series. When $d_1, d_2$ are coprime and fundamental, these Fourier coefficients are explicitly computed in \cite{GZ85}, from which it is clear that they are non-negative.
To compute these Fourier coefficients in general, one needs to evaluate certain local Whittaker integrals. Very general results in this regard have just become available in \cite{YYY18}, which we use here to deduce the non-negativity in Proposition \ref{prop:nonneg}. 

Note that it is crucial that $\varphi_m(j(z_1), j(z_2))$ is the Borcherds product associated to a modular function whose principal part Fourier coefficients are all non-negative. For other modular functions where this is not satisfied, it is very much possible that their CM values are algebraic, integral units (even very often \cite{YYY18})!

\textbf{Acknowledgement.} The author would like to thank Philipp Habegger for helpful comments. We would also like to thank Jan Bruinier for providing us with a preliminary draft of \cite{BEY}, and Tonghai Yang for helpful discussions about \cite{BKY12}, \cite{YY19} and \cite{YYY18}.
Finally, we would like to thank the anonymous referee for helpful comments. 
\section{Higher Green's Function}
The function $\log |j(z_1) - j(z_2)|^2$ is the Green's function for the diagonal on two copies of the modular curve. In \cite{GKZ87}, higher Green's functions were studied. 
For $\Re(s) > 1$, let 
\begin{equation}
  \label{eq:Qs}
Q_{s-1}(t) := \int^\infty_0 (t + \sqrt{t^2 - 1} \cosh v)^{-s} dv,~ t > 1
\end{equation}
 be the Legendre function of the second kind, which satisfies the ordinary differential equation 
$$
(1 - t^2) Q''(t) - 2t Q'(t) + s(s-1) Q(t) = 0.
$$
From the definition above, we see that for any fixed $s \in (1, \infty)$, the function $Q_{s-1}(t)$ is positive and monotonically decreasing for $t \in (1, \infty)$. 
Define a function $g_s$ on $\Hb^2$ by
\begin{equation}
  \label{eq:gs}
  g_s(z_1, z_2) := -2Q_{s-1} (\cosh \mathrm{d}(z_1, z_2)) = -2 Q_{s - 1} \lp 1 + \frac{|z_1 - z_2|^2}{2y_1y_2} \rp
\end{equation}
for $(z_1, z_2) \in \Hb^2$ with $\mathrm{d}(z_1, z_2)$ the hyperbolic distance between $z_1$ and $z_2$.
By averaging over the $\Gamma (=\PSL_2(\Zb))$-translates of the second variable, we obtain a function 
\begin{equation}
  \label{eq:Gs}
  G_s(z_1, z_2) := \sum_{\gamma \in \Gamma} g_s(z_1, \gamma z_2)
\end{equation}
on $Y(1)^2$ symmetric in $z_1$ and $z_2$.
Easy estimates show that the sum converges absolutely and uniformly on compact subsets of $\Hb^2$ when $\Re(s) > 1$. 
In that case, $G_s(z_1, z_2)$ is an eigenfunction of the hyperbolic Laplacian $\Delta_{z_i, 0}$ with eigenvalue $s(1-s)$, where
\begin{equation}
\label{eq:diffops}
\begin{split}
  R_{z, \kappa} &:= 2i \partial_z + \frac{\kappa}{y}, \;
L_{z, \kappa} := -2iy^2 \partial_{\overline{z}}, \\
\Delta_{z, \kappa} &:= - R_{z, \kappa - 2} L_{z, \kappa} = - L_{z, \kappa + 2} R_{z, \kappa} - \kappa = -y^2 \lp \partial_x^2 + \partial_y^2 \rp + i\kappa y (\partial_x + i \partial_y).
\end{split}
\end{equation}
for $\kappa \in \Zb$.
  When $s = 1$, the sum in \eqref{eq:Gs} does not converge absolutely any more. One can however analytically continue $G_s(z_1, z_2)$ to $s = 1$, where it will have a pole. After eliminating the pole using the real-analytic Eisenstein series, one obtains the function $2 \log|j(z_1) - j(z_2)|$ (see \cite[Prop.\ 5.1]{GZ85}).
  Therefore, we will define
  \begin{equation}
    \label{eq:G1}
    G_1(z_1, z_2) := -2 \log |j(z_1) - j(z_2)|
  \end{equation}
  for convenience later.

When $s = k \ge 1$ is an integer, the Legendre function $Q_{k-1}(t)$ has the form
\begin{equation}
  \label{eq:Qk}
Q_{k-1}(t) = \frac{P_{k-1}(t) }{2} \log \frac{t + 1}{t-1} - R_{k-1}(t)  ,
\end{equation}
where $P_{k-1}(t)$ is the $(k-1)^\mathrm{st}$ Legendre polynomial given by
\begin{equation}
  \label{eq:Pn}
  P_{k-1}(t) = \frac{1}{2^{k-1}}\sum_{\ell = 0}^{k-1} \binom{k-1}{\ell}^2 (t-1)^{k-1-\ell}(t+1)^{\ell}
\end{equation}
and $R_{k-1}(t)$ is a unique polynomial.
For $k = 1, 3, 5, 7$, they are given by
\begin{equation}
  \label{eq:Ppol}
  \begin{split}
    P_0(t) &= 1,~ R_2(t) =  0,\\ 
  P_2(t) &= \frac{3t^2 - 1}{2},~ R_2(t) =  \frac{3}{2}t,\\
  P_4(t) &= \frac{35t^4 - 30t^2 + 3}{8},~ R_4(t) = \frac{35}{8}t^3 - \frac{55}{24}t,\\
  P_6(t) &= \frac{231t^6 - 315t^4 + 105t^2 - 5}{16},~ R_6(t) = \frac{231}{16}t^5 - \frac{119}{8} t^3 + \frac{231}{80}.  
  \end{split}
\end{equation}
For $m \in \Nb$, we can let the $m$\tth Hecke operator $T_m$ act on one of $z_1$ and $z_2$ to define
\begin{equation}
  \label{eq:Gsm}
  G^m_{k}(z_1, z_2) := \sum_{\gamma \in \Gamma \backslash \Gamma_m} G_k(z_1, \gamma z_2).
\end{equation}
It has logarithmic singularity along the divisor
\begin{equation}
  \label{eq:Tc}
  \mathcal{T}_m := \{(z, \gamma \cdot z): z \in Y(1), \gamma \in \Gamma_m\} \subset Y(1)^2.
\end{equation}
Given any weakly holomorphic modular form $f \in M_{2-2k}^!$ with the Fourier expansion $ f(\tau) = \sum_{1 \le m \le m_0} c_f(-m)q^{-m} + O(1)$, we can define
\begin{equation}
  \label{eq:Gf}
  G_f(z_1, z_2) := \sum_{1 \le m \le m_0} c_f(-m) m^{k-1} G^m_k(z_1, z_2). 
\end{equation}
Note that if $k = 1$ and $f(\tau) = J_m(\tau) := q^{-m} + O(q)$ is the unique modular function in $M_0^!$, then $G_f(z_1, z_2) = -2 \log |\varphi_m(j(z_1), j(z_2))|$. 
In this case, Borcherds showed that $\log|G_f|$ is the the regularized theta lift of $f$ \cite{Borcherds98}. 
The extension of this result to all $k \ge 1$ can be stated as follows (see \cite{Bruinier02, Viazovska11}).
\begin{prop}[Prop.\ 4.2 in \cite{Li18}]
  \label{prop:Glift}
For an integer $r \ge 0$ and $\tau = u + iv \in \Hb$, denote $R^r_{\tau, \kappa} := R_{\tau, \kappa + 2r -2} \circ R_{\tau, \kappa + 2r -4} \circ \dots \circ R_{\tau, \kappa}$. 
Then for any $f \in M^!_{2-2k}$ and $z_1, z_2 \in \Hb$ with $k \ge 1$, we have
\begin{equation}
  \label{eq:Glift}
  G_f(z_1, z_2) = (4 \pi)^{1-k} \lim_{T \to \infty} \int^{}_{\Fc_T} f(\tau) (R^{k-1}_{\tau, 0}\Theta_L)(\tau; z_1, z_2) d\mu(\tau),
\end{equation}
where $\Fc_T := \Fc \cap \{\tau = u + iv \in \Hb: v \le T\}$ is the truncated fundamental domain and $\Theta_L(\tau; z_1, z_2)$ is the Siegel theta function associated to the  unimodular lattice $L = M_2(\Zb)$.
Furthermore, $G_f$ has logarithmic singularity along $\mathcal{T}_m$ for $m \ge 1$ if and only if $c_f(-m) \neq 0$. 
\end{prop}

To evaluate the theta integral above, one can try to find a preimage of $R^{k-1}_{\tau, 2-2k}\Theta_L$ under the lowering operator $L_\tau$. This is possible for $k = 1$ when one averages over suitable toric orbit of CM points $(\zm_1, \zm_2)$, as we will see in the next section. For odd $k \ge 2$, one can apply the following operator $\Cc_{k-1}$ due to Cohen to obtain the desired preimage. 

\begin{prop}
  \label{prop:Cohen}
For a real-analytic function $F(\tau)$ on $\Hb^2$, suppose there exists a real-analytic function $\tilde{F}(\tau_1, \tau_2)$ on $\Hb^2$ such that it is harmonic in $\tau_1, \tau_2$, and satisfies
$$
L_\tau \tilde{F}(\tau, \tau) = F(\tau).
$$
Then for $k \ge 1$ odd, we have $L_\tau(\Cc_{k-1} \tilde{F}) = (4\pi)^{1-k} R^{k-1}_{\tau, 0}(F)$, where $\Cc_{k-1}$ is the Cohen operator defined by
$$
(\Cc_{k-1} \tilde{F})(\tau) := (-2\pi i)^{1 - k}\sum_{\ell = 0}^{k-1} (-1)^\ell \binom{k-1}{\ell}^2 \partial_{\tau_1}^{\ell} \partial_{\tau_2}^{k-1-\ell} \tilde{F}(\tau_1, \tau_2) \mid_{\tau_1 = \tau_2 = \tau}.
$$
and satisfies
$$
(\Cc_{k-1} \tilde{F})\mid_{2k} \gamma = 
\Cc_{k-1} (\tilde{F} \mid_{(1, 1)} (\gamma, \gamma))
$$
for all $\gamma \in \mathrm{SL}_2(\Rb)$.
\end{prop}

\begin{rmk}
  The definition above is unchanged if one replaces $\partial_{\tau_i}^r$ with $(2 i)^{-r} R_{\tau_i, 1}^{r}$. The proposition can then be checked easily using \eqref{eq:diffops}.
\end{rmk}

With some technical conditions on the discriminants, the averaged value of $G_f(z_1, z_2)$ at CM points was studied in \cite{GKZ87}, extending results in \cite{GZ85} and \cite{GZ86}. As in the case $k = 1$, these averaged values are logarithms of rational numbers, which can be factored explicitly. 
In view of \cite{BKY12}, these results can be put in the framework of arithmetic intersection on Hilbert modular surface, and the technical conditions can be removed. 
We will recall these results in the next section.

\section{Theorems of Gross-Zagier, Gross-Kohnen-Zagier and Their generalizations}
Let $\zm \in Y(1)$ be a CM point of discriminant $d < 0$, and $\Oc_d, K, H_d$ as in the introduction.
For an element $t$ in the finite ideles $\Ab_{K, f}^\times$ of $K$, let $\sigma_t \in \Gal(H_d/K)$ be the associated element via the Artin map. Then $\sigma_t$ acts naturally on $\zm$.

Given two CM points $\zm_i \in Y(1)$ with discriminant $d_i$, one can realize $(\zm_1, \zm_2) \in Y(1)^2$ as small/big CM points (depending on whether $D:=d_1d_2$ is a perfect square or not) in the sense of \cite{Schofer} and \cite{BKY12}. 
The averaged values of $G_f$ at these CM points can be expressed in terms of Fourier coefficients of incoherent Eisenstein series. 
In this section, we will recall these results. 

\subsection{Incoherent Eisenstein Series}
First, we quickly recall the incoherent Eisenstein series in the sense of Kudla \cite{Kudla97} (see e.g.\ section 4 of \cite{BKY12}).
Let $F$ be a totally real field of degree $n_0$, $E/F$ a quadratic, CM extension, and $W = E$ an $F$-quadratic space with quadratic form $Q_F(x) = \alpha x \overline{x}$ for some $\alpha \in F^\times$. Denote $\{\sigma_j: 1 \le j \le n_0\}$ the real embeddings of $F$ and $(V, Q)$ the restriction of scalar of $(W, Q_F)$ to $\Qb$. 
If $\alpha$ is chosen such that $\sigma_{n_0}(\alpha) < 0$ and $\sigma_j(\alpha) > 0$ for $1 \le j \le n_0-1$, then $(V, Q)$ has signature $(n, 2)$. 

Fix an additive adelic character $\psi$ of $\Qb$ and denote $\psi_F = \psi \circ \Tr_{F/\Qb}$. 
Associate to it is a Weil representation $\omega = \omega_{\psi_F}$ of $\SL_2(\Ab_F)$ on $S(W(\Ab_F)) = S(V(\Ab_\Qb))$. 
Let $\chi_{}$ be the quadratic Hecke character associated to $E/F$. 
For any element $\Phi$ in the principal series representation $I(s, \chi)$, one can define a Hilbert Eisenstein series
$$
E(g, s, \Phi) :=  \sum_{\gamma \in B_F \backslash \SL_2(F)} \Phi(\gamma g, s),~ g \in \SL_2(\Ab_F)
$$
with $\Re(s) \gg 0$, and analytically continue it to $s \in \Cb$. 
At the infinite places, we choose $\Phi$ to be the unique eigenvector of $\SL_2(\Rb)$ of weight 1. 

At the finite place, one can use any $\phi \in S(V(\Ab_{\Qb, f}))$ to construct a section. Using these information, we can define a Hilbert Eisenstein series $E(\vec{\tau}, s, \phi)$, which is a real-analytic Hilbert modular form of parallel weight 1 in $\vec{\tau} = (\tau_1, \dots, \tau_{n_0}) \in \Hb^{n_0}$ (see equation (4.4) in \cite{BKY12}). We can further normalize it by
$$
E^*(\vec{\tau}, s, \phi) = \Lambda(s + 1, \chi) E(\vec{\tau}, s, \phi),
$$
with $\Lambda(s, \chi)$ the completed $L$-function associated to $\chi$. 
Usually, we will take $\phi = \phi_\mu$ the characteristic function of $(L + \mu) \otimes \hat{\Zb}$ for $\mu \in L^\vee/L$ and $L \subset V$ some even integral lattice. In that case, we use $E^*(\vec{\tau}, s, L)$
to denote the vector-valued modular form $\sum_{\mu \in L^\vee/L} E^*(\vec{\tau}, s, \phi_\mu) \ef_\mu $.

Because of the choice of the section at the infinite places, this Eisenstein series is incoherent in the sense of Kudla \cite{Kudla97}, and vanishes at $s = 0$. 
Its derivative at $s = 0$ is a real-analytic Hilbert modular form of parallel weight $1$. 
For a totally positive $t \in F^\times$, denote its $t$\tth Fourier coefficient by $a(t, \phi)$. 
By Prop.\ 4.6 in \cite{BKY12}, one can write 
\begin{equation}
  \label{eq:am}
a_m(\phi) := \sum_{t \in S_m} a(t, \phi) = \sum_{p} a_{m, p}(\phi) \log p,  
\end{equation}
with $S_m :=  \{t \in F^\times: t \gg 0, \Tr(t) = m\}$ and $a_{m, p}(\phi)$ in the subfield of $\Cb$ generated by the values $\phi(x)$ for $x \in V(\Ab_{\Qb, f})$. 
When $\phi = \otimes \phi_{\pf}$ is factorizable, this is given by product of values of local Whittaker functions, which have been explicitly calculated in many cases (see e.g.\ \cite{Yang05}, \cite{KY10}, \cite{YYY18}). 
Using these explicit formula, we can say something more refined about $a(t, \phi)$ when $\phi = \phi_\mu$.

\begin{prop}
  \label{prop:nonneg}
For any even, integral lattice $L \subset V$ and $\mu \in L^\vee/L$, we have
\begin{equation}
  \label{eq:nonneg}
  - a(t, \phi_\mu) \ge 0
\end{equation}
whenever $t \in F^\times$ is totally positive. Furthermore for any $m > 0$, the coefficient $a(t, \phi_\mu) = 0$ for all but finitely many $t \in S_m$.
\end{prop}

\begin{proof}
  Since any two lattices in $V$ of full rank are commensurable, we can find a positive integer $c$ such that $c\Oc_E$ is a sublattice of $L$ and $c\Oc_E \subset L \subset L^\vee \subset (c\Oc_E)^\vee$.
Let $\varpi: L^\vee/c\Oc_E \to L^\vee/L$ be the natural projection.
We can then write 
$$
\phi_\mu = \sum_{\mu' \in L^\vee/c\Oc_E,~ \varpi(\mu') = \mu} \phi'_{\mu'}
$$ 
with $\phi'_{\mu'}$ the characteristic function of $c\Oc_E + \mu'$. 
Therefore, it suffices to prove the claim for $L = c\Oc_E$. 
After scaling, we can then suppose $L = \Oc_E$ and $0 \neq \alpha \in \Oc_F$, as $L$ is even integral.
Then the section $\phi = \phi_\mu = \otimes \phi_{\pf} $ is factorizable and the coefficient $a(t, \phi)$ is given by
$$
a(t, \phi) = -2^{n_0} \frac{W^{*,\prime}_{t, \pf_0}(0, \phi_{\pf_0})}{\gamma(W_{\pf_0})} \prod_{\pf \nmid \pf_0 \infty} \frac{W^{*}_{t, \pf}(0, \phi_\pf)}{\gamma(W_{\pf})}
$$
if the 'Diff' set of Kudla $\mathrm{Diff}(W, t)$ consists of just a finite prime $\pf_0$, and vanishes otherwise (see Prop.\ 2.7 in \cite{YY19}). Here $W^*_{t, \pf}(s, \phi)$ is the normalized local Whittaker function (see equation (2.25) in \cite{YY19}) and $\gamma(W_\pf)$ is the local Weil index. 
These local Whittaker functions have been computed explicitly in all cases in the appendix of \cite{YYY18}, from which we know that 
$$
\frac{W^{*}_{t, \pf}(0, \phi_\pf)}{\gamma(W_{\pf})} \ge 0,~ 
\frac{W^{*,\prime}_{t, \pf}(0, \phi_{\pf})}{\gamma(W_{\pf})} \ge 0.
$$
for all totally positive $t \in F^\times$. 
Furthermore, there is a positive integer $b$ depending on $L, E, F$ such that $a(t, \phi) = 0$ for all $t \not\in b^{-1} \Oc_F$. 
This proves the second claim.
\end{proof}
\subsection{Big CM Points}
When $D$ is not a perfect square, $F = \Qb(\sqrt{D})$ is a real quadratic field. 
For $i = 1, 2$, let $z_i  \in \Hb \cap K_1K_2$ be a representative of $\zm_i$, and $\af_i := \Zb  + \Zb z_i \subset K_i$ the corresponding $\Zb$-module. 
Denote $W = K_1K_2$ the $F$-quadratic space with the quadratic form $Q_F(x) = \frac{x\bar{x}}{ \sqrt{D}}$. 
We can identify $(V, Q)$ with the rational quadratic space $(M_2(\Qb), \det)$ via the map
\begin{align*}
   \sum_{i = 1}^4 &x_i e_i \in V \mapsto \pmat{x_3}{x_1}{x_4}{x_2} \in M_2(\Qb),\\
e_1 &:= 1,~
e_2 := -\overline{z_1},~
e_3 := z_2,~
e_4 := -\overline{z_1}z_2,
\end{align*}
under which $\overline{\af_1}\af_2$ is mapped to the unimodular lattice $L = M_2(\Zb)$. 
Define a torus
$$
T(R) := \{(t_1, t_2) \in (K_1 \otimes_\Qb R)^\times \times (K_2 \otimes_\Qb R)^\times: t_1\overline{t_1} = t_2 \overline{t_2} \}
$$
for any $\Qb$-algebra $R$.
It embeds into the algebraic group 
\begin{equation}
  \label{eq:H}
\mathrm{GSpin_V}(R) \cong \mathrm{H}(R) :=  \{(g_1, g_2) \in \mathrm{GL}_2(R) \times  \mathrm{GL}_2(R): \det(g_1) = \det(g_2)\}  
\end{equation}
via the map $\iota = (\iota_1, \iota_2): T \to \mathrm{H}$, where
$$
(e_1, e_2)  \iota_1(r)  = ( r e_{1}, r e_2), \quad  \iota_2(r) (e_{3}, e_1)^t = ( \bar{r} e_{3}, \bar{r} e_1)^t.
$$
Denote $K_T := \iota^{-1}(\mathrm{H}(\hat{\Zb})) \subset T(\Ab_f)$ a compact subgroup preserving the torus, and $K_{T, i} := \iota_i^{-1}(\mathrm{H}(\hat{\Zb})) \subset \Ab_{K_i, f}^\times$ for $i = 1, 2$. Then $K_{T, i} = \hat{\Oc}_{d_i}$ and $K_i^\times \backslash \Ab_{K_i, f}^\times / K_{T, i}$ is just the class group of the order $\Oc_{d_i}$, which we denote by $\Cl(d_i)$.
As in Lemma 3.5 in \cite{YY19}, there is a natural injection 
$$
p': T(\Qb) \backslash T(\Ab_f) /K_T \to \Cl(d_1) \times \Cl(d_2)
$$ 
sending $[(t_1, t_2)]$ to $([t_1], [t_2])$. 
On the other hand,  $H = H_1H_2$ implies that the natural map 
$$
p'':\Gal(H/\Qb) \to \Gal(H_1/\Qb) \times \Gal(H_2/\Qb)
$$
sending $\sigma$ to $(\sigma\mid_{H_1}, \sigma\mid_{H_2})$ is injective.
Since $H_j/K_j$ is a ring class field, the Galois group $\Gal(H_j/\Qb)$ is a generalized dihedral group. 
This gives us the following result.
\begin{lemma}[Theorem 8.3.12 in \cite{Cohn85}]
  \label{lemma:abelian}
  Let $H_0 := H_1 \cap H_2 \subset H$ be the intersection of ring class fields. Then $H_0/\Qb$ is abelian.
\end{lemma}
\begin{proof}
  We collect the proof from \cite{Cohn85} here. 
By replacing $H_0$ with $H_0K_1K_2$, we can suppose that $H_0$ contains $K_1K_2$. Clearly, $H_0 \subset H_j$ is abelian over $K_j$ for $j = 1, 2$. Therefore, it is also abelian over $K_1K_2$. Since $\Gal(H_0/\Qb)$ is a quotient of $\Gal(H_j/\Qb)$, which is generalized dihedral, we can find $\sigma_j \in \Gal(H_0/\Qb)$ of order 2 such that $\sigma_j h \sigma_j^{-1} = h^{-1}$ for all $h \in \Gal(H_0/K_j)$. 
That means $\Gal(H_0/\Qb)$ is generated by the abelian group $\Gal(H_0/K_1K_2)$ and the elements $\sigma_1, \sigma_2$. Since $\Gal(H_0/K_j)$ is abelian, we know that $\sigma_j$ commutes with $\Gal(H_0/K_1K_2)$. But $\sigma_1, \sigma_2$ also commute since
$$
\sigma_1 \sigma_2 \sigma_1^{-1} = \sigma_2^{-1} = \sigma_2. 
$$
This finishes the proof.
\end{proof}
After identifying $\Cl(d_i)$ with $\Gal(H_i/K_i) \subset \Gal(H_i/\Qb)$ via the Artin map, we can now state the following analogue of Lemma 3.8 of \cite{YY19}.
\begin{prop}
  \label{prop:classgp}
Using the notation introduced above, the image of $p'$ is contained in the image of $p''$, and $p''^{-1} \circ p': T(\Qb) \backslash T(\Ab_f) /K_T \to \Gal(H/K_1K_2)$ is an isomorphism.
\end{prop}
\begin{rmk}
  When $(d_1, d_2) = 1$, the map $p''$ is an isomorphism and we recover Lemma 3.8 in \cite{YY19}.
\end{rmk}

\begin{proof}
The image of $p''$ is exactly given by 
$$
p''(\Gal(H/K_1K_2)) = \{(\sigma_1, \sigma_2) \in \Gal(H_1/K_1) \times \Gal(H_2/K_2): \sigma_1 \mid_{H_0} = \sigma_2 \mid_{H_0} \}.
$$
On the other hand, if $t_i \in K_i^\times \backslash \Ab_{K_i, f}^\times / K_{T, i}$ is associated to $\sigma_i \in \Gal(K_i^{\mathrm{ab}}/K_i)$ via the Artin map, then $t_i\overline{t_i} = \Nm_{K_i/\Qb}(t_i)$ is associated  $\mathrm{res}(\sigma_i) \in \Gal({\Qb}^{\mathrm{ab}}/\Qb)$ with $\mathrm{res}: \Gal(K_i^{\mathrm{ab}}/K_i) \to \Gal(\Qb^{\mathrm{ab}}/\Qb)$ the natural restriction map.
So the image of $p'$ is 
$$
\mathrm{Im}(p') = \{(\sigma_1, \sigma_2) \in \Gal(H_1/K_1) \times \Gal(H_2/K_2): \sigma_1 \mid_{\Qb^\mathrm{ab} \cap H_0} = \sigma_2 \mid_{\Qb^\mathrm{ab} \cap H_0} \}.
$$
Since $H_0/\Qb$ is abelian by Lemma \ref{lemma:abelian}, this finishes the proof.
\end{proof}

Now let $Z(W)$ be the big CM point on $Y(1)^2$ associated to $W$ as in \cite{BKY12}. Then the argument in section 3 of \cite{YY19} with the proposition above immediately implies that 
\begin{equation}
  \label{eq:bigCM1}
  \begin{split}
Z(W) = 
\sum_{\sigma \in \Gal(H/K_1K_2)} (\zm_1, \zm_2)^{\sigma} +
(-\overline{\zm_1}, \zm_2)^{\sigma} + 
(\zm_1, -\overline{\zm_2})^{\sigma} +
(-\overline{\zm_1}, -\overline{\zm_2})^{\sigma}.
  \end{split}
\end{equation}
Now, the value at $Z(W)$ of the higher Green's function $G_f$ with $f \in M_{2-2k}^!$ and $k \ge 1$ odd can be explicitly given as a finite sum of Fourier coefficients of certain incoherent Eisenstein series. We state it as follows.
\begin{thm}
\label{thm:GKZ1}
Let $f \in M_{2-2k}^!$ with $k \ge 1$ odd and vanishing constant term if $k = 1$. Suppose $d_1d_2$ is not a perfect square. Then
  \begin{equation}
    \label{eq:GKZ1}
    \begin{split}
      G_f(Z(W)) &= \frac{|Z(W)|}{2\Lambda(0, \chi)} \sum_{m \ge 1} c_f(-m) a_{m}(\phi; k),\\
a_{m}(\phi; k) &:=      m^{k-1}\sum_{t \in S_m} P_{k-1}\lp \frac{t - t'}{m} \rp a_{}(t, \phi),
    \end{split}
  \end{equation}
  where $a_{}(t, \phi)$ is the $t$\tth Fourier coefficient of holomorphic part of the incoherent Eisenstein series ${E^{*,}}'((\tau_1, \tau_2), 0, \phi)$ of parallel weight $(1, 1)$ with $\phi = \phi_{\af_1, \af_2} \in S(V \otimes \Ab_f) = S(W_\Qb \otimes \Ab_f)$ the characteristic function of $\overline{\af_1} \af_2 \otimes \hat{\Zb}$, and $\Lambda(s, \chi)$ the completed $L$-function.
\end{thm}

\begin{rmk}
  \label{rmk:constant}
The constant $\frac{|Z(W)|}{2\Lambda(0, \chi)}$ is explicitly given by $\frac{w_1w_2 [H_1:K_1] [H_2:K_2]}{2h_1h_2[H_0:\Qb]}$
with $w_i := | \Oc_{K_i}^\times|$, $h_i$ the class number of $K_i$ and $H_0 := H_1 \cap H_2$. 
\end{rmk}

\begin{rmk}
  \label{rmk:singularity}
  For any $\gamma \in \Gamma_m$ and CM point $\zm \in Y(1)$ of discriminant $d$, the CM point $\gamma \cdot \zm \in Y(1)$ has discriminant $d'$ such that $dd'$ is a perfect square. Since the singularity of $G_f$ is supported on $\mathcal{T}_m \subset Y(1)^2$, it does not intersect the CM cycle $Z(W)$.
\end{rmk}
\begin{proof}
  When $k = 1$, this is just Theorem 5.2 in \cite{BKY12}.
When $k \ge 2$, this can be easily modified using the Cohen operator in Prop.\ \ref{prop:Cohen} (see e.g.\ Theorem 5.10 in \cite{BEY}).
For the convenience of the readers, we include some details here.
By applying Lemma 4.3 and Prop.\ 4.5 in \cite{BKY12} and Prop.\ \ref{prop:Cohen}, we obtain
$$
(4\pi)^{1-k} \Theta_L(\tau; Z(W)) = 
- \frac{|Z(W)|}{2\Lambda(0, \chi)}
L_\tau \mathcal{C}_{k-1}\lp  {E^{*,}}'((\tau_1, \tau_2), 0, \phi) \rp
$$
for any $k \ge 1$ odd. 
Substituting this into \eqref{eq:Glift} and apply Stokes' Theorem yields (c.f.\ proof of Theorem 5.2 in \cite{BKY12} and Theorem 5.10 in \cite{BEY})
$$
G_f(Z(W)) = 
- \frac{|Z(W)|}{2\Lambda(0, \chi)}
 \lim_{T \to \infty} \int_{ \mathcal{F}_T} 
d ( f\cdot \mathcal{C}_{k-1}(   {E^{*,}}' ) d \tau)
= \frac{|Z(W)|}{2} \mathrm{CT}(f \mathcal{C}_{k-1}(\mathcal{E}^+_L)),
$$
where 
$$
\mathcal{E}_L^+(\tau_1, \tau_2) = \frac{1}{\Lambda(0, \chi)} \sum_{m \in \Nb,~ t \in S_m} a(t, \phi) \ebf(t \tau_1 + t' \tau_2)
$$ 
is the holomorphic part of $E'((\tau_1, \tau_2), 0, \phi) = \frac{{E^{*,}}'((\tau_1, \tau_2), 0, \phi)}{\Lambda(0, \chi)}$.
We are now done after applying the identity
\begin{align*}
  \mathcal{C}_{k-1} \ebf(t \tau_1 + t' \tau_2) &= 
\lp \sum_{\ell = 0}^{k-1} (-1)^\ell \binom{k-1}{\ell}^2  t^\ell (t')^{k-1-\ell}\rp 
\ebf((t + t') \tau)\\
&= P_{k-1} \lp \frac{t - t'}{t + t'} \rp (t + t')^{k-1} \ebf((t + t') \tau),
\end{align*}
which comes from the definitions of $\mathcal{C}_{k-1}$ in Prop.\ \ref{prop:Cohen}, $P_{k-1}$ in \eqref{eq:Pn}, and the fact that $k-1$ is even.
\end{proof}

\subsection{Small CM Points}
Suppose $D = d_1d_2$ is a perfect square. Then $K_1$ and $K_2$ are the same quadratic field $K$ of fundamental discriminant $d < 0$.
Let $z_i = x_i + i y_i  \in \Hb \cap K$ be a representative of $\zm_i$.
Then $y_1y_2 \in \Qb$ and the CM point $(z_1, z_2)$ arises from a rational splitting of $(V, Q) = (M_2(\Qb), \det)$. To be more precise, let $W \subset V$ be the rational, negative 2-plane spanned by the rational vectors $\Re Z(z_1, z_2) , {\sqrt{|d|}} \Im Z(z_1, z_2)  \in V$, where
\begin{equation}
  \label{eq:Z}
  Z(z_1, z_2) :=  \pmat{z_1}{-z_1z_2}{1}{-z_2}
\in V(\Cb).
\end{equation}
Then the element $W \otimes \Rb$ in the Grassmannian of $V( \Rb)$ corresponds to the points $z_0^+ := (z_1, z_2) \in \Hb^2$ and $z_0^- := (\overline{z_1}, \overline{z_2}) \in (\Hb^-)^2$.

On the level of lattice, denote $L = M_2(\Zb) \subset V$ and consider a finite index sublattice
$$
L_0 := 
L_+ \oplus L_- \subset L,~ 
L_+ := L \cap W^\perp,~
L_- := L \cap W.
$$
The holomorphic theta function 
$$
\theta_{L_+}(\tau) := 
 \sum_{\mu_1 \in L_+^\vee/L_+} \ef_{\mu_1} \theta_{L_+, \mu_1}(\tau),~
\theta_{L_+, \mu_1}(\tau) := \sum_{\lambda \in L_+ + \mu_1} q^{Q(\lambda)}
$$
is a vector-valued holomorphic modular form of weight 1 with respect to the Weil representation $\rho_{L_+}$. 
On the other hand, the incoherent Eisenstein series 
$$
{E^{*,}}'(\tau, 0, L_-)  = \sum_{\mu_2 \in L^\vee_-/L_-}
{E^{*,}}'(\tau, 0, \phi_{\mu_2}) \ef_{\mu_2}
$$ 
is a real-analytic, elliptic modular form of weight 1 with respect to the Weil representation $\rho_{L_-}$. 
Their tensor product is a real-analytic modular form of weight 2 with respect to $\rho_{L_0}$. 
The following function on $\Hb^2$
\begin{equation}
  \label{eq:tensor}
  \tilde{F}(\tau_1, \tau_2; L_+, L_-) := 
\sum_{\mu = (\mu_1, \mu_2) \in L/L_0 \subset L_0^\vee/L_0} \theta_{L_+, \mu_1}(\tau) {E^{*,}}'(\tau, 0, \phi_{\mu_2}).
\end{equation}
satisfies $\tilde{F} \mid_{1, 1}(\gamma, \gamma) = \tilde{F}$ for all $\gamma \in \Gamma$.

The torus $T_W$, whose $R$-points are $(R \otimes K)^\times$ for any $\Qb$-algebra $R$, is embedded into the algebraic group $\mathrm{H}$ defined in \eqref{eq:H} through the map $\iota = (\iota_1, \iota_2): T_W \hookrightarrow \mathrm{H}$ defined by
\begin{align*}
\iota_i \lp a + b\frac{y_1y_2}{\sqrt{d}} \rp := a \pmat{1}{}{}{1} + (-1)^{i+1} b \frac{y_{3-i}}{\sqrt{|d|}} \pmat{x_i}{-|z_i|^2}{1}{-x_i}
\end{align*}
for $i = 1, 2$.
Simple routine calculations then show that $\iota^{-1}(\mathrm{H}(\Zb)) = \Oc_{d'} \subset K$ with $d' := -d_1 d_2/(d_1, d_2)^2 < 0$ the largest negative discriminant divisible by $d_1$ and $d_2$. 
The class group $T_W(\Qb) \backslash T_W(\Ab_f)/\hat{\Oc}_{d'}$ is just the class group of the order $\Oc_{d'}$. 
Therefore, the small CM 0-cycle $Z(W) := T_W(\Qb) \backslash (\{z_0^\pm\} \times T_W(\Ab_f)/\hat{\Oc}_{d'})$ on $Y(1)^2$ becomes 
\begin{equation}
  \label{eq:ZW2}
Z(W) = 
\sum_{\sigma \in \Cl(d')} (\zm_1, \zm_2)^\sigma + (-\overline{\zm_1}, -\overline{\zm_2})^\sigma.  
\end{equation}
We can now apply Theorem 1.1 in \cite{Schofer} to give a formula for $G_f(Z(W))$ when $f \in M^!_0$, and generalize it to higher Green's function as in Theorem \ref{thm:GKZ1}. 
\begin{thm}
\label{thm:GKZ2}
  Let $f \in M_{2-2k}^!$ with $k \ge 1$ odd and vanishing constant term if $k = 1$. Suppose the singularity of $G_f$ does not intersect $Z(W)$. Then 
  \begin{equation}
    \label{eq:GKZ2}
G_f(Z(W)) = \frac{|Z(W)|}{\Lambda(0, \chi)} \sum_{m \ge 1} c_f(-m)     \kappa(m; k) 
  \end{equation}
where $\kappa(m; k)$ is the $m$\tth Fourier coefficient of $\Cc_{k-1} \tilde{F}(\tau_1, \tau_2; L_+, L_-)$.
\end{thm}

By Remark \ref{rmk:singularity}, it is possible for $Z(W)$ to intersect $\mathcal{T}_m$.
We will give a simple criterion to see when this happens.

\begin{lemma}
  \label{lemma:intersection}
  For $m \ge 1$, the CM cycle $Z(W)$ intersects the divisor $\mathcal{T}_m$ if and only if there exists $\lambda \in L_+$ such that $Q(\lambda) = m$, i.e.\ the $m$\tth Fourier coefficient of $\theta_{L_+, 0}(\tau)$ is positive.
\end{lemma}
\begin{proof}
  The divisor $\mathcal{T}_m \subset Y(1)^2$ is an example of special cycle. 
In particular, its preimage in $\Hb^2$ is the following $\Gamma \times \Gamma$-invariant set
$$
\left\{
(z_1', z_2') \in \Hb^2: \text{ there is } \lambda \in L \text{ such that } Z(z_1', z_2') \perp \lambda \text{ and } \det(\lambda) = m
\right\}.
$$
Note if $Y(1)^2$ is replaced by a Hilbert modular surface, then the analogue of $\mathcal{T}_m$ is the Hirzebruch-Zagier divisor.
From this description, we know that $Z(W)$ intersects $\mathcal{T}_m$ if and only if there exists $\lambda \in L$ satisfying $Q(\lambda) = m$ and $\lambda \perp Z(z_1, z_2)$, i.e.\ $\lambda \in W^\perp \cap L = L_+$. 
The lemma is now clear.
\end{proof}
We can now give an explicit expression for $\kappa(m ;k)$. 
\begin{prop}
  \label{prop:FE}
Suppose the $n_1^{\mathrm{th}}$ Fourier coefficient of $\theta_{L_+, \mu_1}(\tau)$ is $b(n_1, \mu_1)$ for $n_1 \ge 0, 
\mu_1 \in L^\vee_+/L_+$. Then $\kappa(m; k)$ is given by
$$
\kappa(m; k) = 
 m^{k-1}
\lp b(m, 0)a(0, \phi_0)P_{k-1}(1) + 
\sum_{\begin{subarray}{c} \mu = (\mu_1, \mu_2) \in L/L_0\\ n_1 \ge 0,~ n_2 > 0\\ n_1 + n_2 = m\end{subarray}}
b(n_1, \mu_1) a(n_2, \phi_{\mu_2})
P_{k-1} \lp \frac{n_1 - n_2}{m} \rp \rp,
$$
where $a(n_2, \phi_{\mu_2})$ is the $n_2^\text{th}$ Fourier coefficient of the incoherent Eisenstein series ${E^{*,}}'(\tau, 0, \phi_{\mu_2})$.
Furthermore, $c_f(-m)b(m, 0) = 0$ for all $m \ge 1$ if and only if the singularity of $G_f$ does not intersect $Z(W)$. 
\end{prop}

\begin{proof}
From the definition of $\Cc_{k-1}$ and \eqref{eq:tensor}, one can derive the formula for $\kappa$ with sum over $n_i \ge 0$.
If $(n_1, n_2) = (m, 0)$, then we know from Definition 2.16 in \cite{Schofer} that $a(0, \phi_{\mu_2}) = 0$ for all $\mu_2 \neq 0 \in L_-^\vee/L_-$.
This proves the first claim.
The second part follows from Prop.\ \ref{prop:Glift} and Lemma \ref{lemma:intersection}.
\end{proof}
\section{Proof of Theorem \ref{thm:main}}
\label{sec:proof}
When $d_1d_2$ is not a perfect square, Prop.\ \ref{prop:classgp} and Theorem \ref{thm:GKZ1} imply that\begin{align*}
&2 \log \Nm_{H/\Qb} \varphi_m(j(\zm_1) - j(\zm_2)) \\
= & 
2 \log \Nm_{H/K} \lp \varphi_m(j(\zm_1) - j(\zm_2))
\varphi_m(j(-\overline{\zm_1}) - j(\zm_2))
\varphi_m(j(\zm_1) - j(-\overline{\zm_2}))
\varphi_m(j(\zm_1) - j(-\overline{\zm_2}))\rp\\
 =& -G_{J_m}(Z(W)) =   -G^m_1(Z(W)) =   
\frac{|Z(W)|}{2\Lambda(0, \chi)}
 \sum_{t \in S_m} -a(t, \phi).
\end{align*}
For $k = 3, 5, 7$, the space $M^!_{2-2k}$ has a basis $\{f_{k, m}(\tau) = q^{-m} + O(1): m \ge 1\}$ since the space of cusp forms of weight $2k$ is trivial. 
We can specialize Theorem \ref{thm:GKZ1} to $f = f_{k, m}$ to obtain
$$
-m^{1-k} G_{f_{k, m}}(Z(W)) = 
-G_{k}^m(Z(W)) = 
 \frac{|Z(W)|}{2\Lambda(0, \chi)}
\sum_{t \in S_m}
 P_{k-1}\lp \frac{t - t'}{m} \rp (-a_{}(t, \phi)),
$$
for $k = 3, 5, 7$ and any $m \ge 1$.

By Prop.\ \ref{prop:nonneg}, we can easily deduce
\begin{equation}
\label{eq:bound1}
  \begin{split}
2     \log |\Nm_{H/\Qb}\varphi_m(j(\zm_1), j(\zm_2))|   &\ge
\frac{|Z(W)|}{2\Lambda(0, \chi)}
  \sum_{t \in S_m}
  P_{k-1} \lp \frac{t - t'}{m}  \rp (-a(t, \phi))\\
&= -G^m_k(Z(W))
  \end{split}
\end{equation}
since $P_{k-1}(r) \le 1$ for all $r \in[-1, 1]$.
To bound the term $- G_k^m(Z(W))$ from below, we can first apply its definition to write
\begin{align*}
  - G_k^m(Z(W)) &=  \sum_{(z_1, z_2) \in Z(W)} \sum_{\gamma \in \Gamma \backslash \Gamma_m} - G_k(z_1, \gamma z_2) =
  \sum_{(z_1, z_2) \in Z(W) \cap \Fc^2,~ \gamma \in \Gamma_m}
2 Q_{k-1}(\cosh \mathrm{d}(z_1, \gamma z_2) ).
\end{align*}
If we denote $\mathrm{d}_2$ the distance on $\Hb^2$ associated to the product Riemannian metric, then 
\begin{align*}
2 \mathrm{d}_2((z_1, z_2), (z, \gamma z))^2 &= 
2 \lp \mathrm{d}(z_1, z)^2 + \mathrm{d}(z_2, \gamma z)^2 \rp 
= 2 \lp \mathrm{d}(\gamma z_1, \gamma z)^2 + \mathrm{d}(z_2, \gamma z)^2 \rp \\
&\ge (\mathrm{d}(\gamma z_1, \gamma z) + \mathrm{d}( z_2, \gamma z))^2  \ge \mathrm{d}(\gamma z_1, z_2)^2
\end{align*}
for any $\gamma \in \mathrm{PSL}_2(\Rb)$ by the triangle inequality.
Therefore for any $\epsilon > 0$, $(z_1, z_2) \in Z(W)\cap \Fc^2 \cap T_{m, \epsilon} \subset \Hb^2$ implies that there exists $\gamma \in \Gamma_m$ such that $\mathrm{d}(z_1, \gamma z_2) < \sqrt{2} \epsilon$, i.e.\ $Q_{k-1}(\cosh \mathrm{d}(z_1, \gamma z_2)) > Q_{k-1}(\cosh(\sqrt{2} \epsilon))$.
Combining this with \eqref{eq:bound1} and set $k = 3$ finishes the proof.

If $d_1d_2$ is a perfect square, the same argument goes through with Theorem \ref{thm:GKZ1} replaced by Theorem \ref{thm:GKZ2} and Prop.\ \ref{prop:FE}. 
Note that $\varphi_m(j(\zm_1), j(\zm_2)) \neq 0$ is equivalent to that the cycle $Z(W)$ does not intersect the singularity of $G_{J_m}$, which has the same support as that of $G_{f_{k, m}}$. 
This finishes the proof of Theorem \ref{thm:main}.

\bibliography{singular}{}

\providecommand{\bysame}{\leavevmode\hbox to3em{\hrulefill}\thinspace}
\providecommand{\MR}{\relax\ifhmode\unskip\space\fi MR }
\providecommand{\MRhref}[2]{%
  \href{http://www.ams.org/mathscinet-getitem?mr=#1}{#2}
}
\providecommand{\href}[2]{#2}
\begin{thebibliography}{10}

\bibitem{BHK18}
Yuri Bilu, Philipp Habegger, and Lars K\"{u}hne, \emph{No {S}ingular {M}odulus
  {I}s a {U}nit}, Int. Math. Res. Not. IMRN (2020), no.~24, 10005--10041.
  \MR{4190395}

\bibitem{Borcherds98}
Richard~E. Borcherds, \emph{Automorphic forms with singularities on
  {G}rassmannians}, Invent. Math. \textbf{132} (1998), no.~3, 491--562.
  \MR{1625724 (99c:11049)}

\bibitem{Bruinier02}
Jan~H. Bruinier, \emph{Borcherds products on {O}(2, {$l$}) and {C}hern classes
  of {H}eegner divisors}, Lecture Notes in Mathematics, vol. 1780,
  Springer-Verlag, Berlin, 2002. \MR{1903920 (2003h:11052)}

\bibitem{BEY}
Jan~Hendrik Bruinier, Stephan Ehlen, and Tonghai Yang, \emph{Higher {G}reen
  functions and their {CM} values}, Invent. Math. (to appear).

\bibitem{BKY12}
Jan~Hendrik Bruinier, Stephen~S. Kudla, and Tonghai Yang, \emph{Special values
  of {G}reen functions at big {CM} points}, Int. Math. Res. Not. IMRN (2012),
  no.~9, 1917--1967. \MR{2920820}

\bibitem{Cohn85}
Harvey Cohn, \emph{Introduction to the construction of class fields}, Cambridge
  Studies in Advanced Mathematics, vol.~6, Cambridge University Press,
  Cambridge, 1985. \MR{812270}

\bibitem{GKZ87}
B.~Gross, W.~Kohnen, and D.~Zagier, \emph{Heegner points and derivatives of
  {$L$}-series. {II}}, Math. Ann. \textbf{278} (1987), no.~1-4, 497--562.
  \MR{909238 (89i:11069)}

\bibitem{GZ85}
Benedict~H. Gross and Don~B. Zagier, \emph{On singular moduli}, J. Reine Angew.
  Math. \textbf{355} (1985), 191--220. \MR{772491 (86j:11041)}

\bibitem{GZ86}
\bysame, \emph{Heegner points and derivatives of {$L$}-series}, Invent. Math.
  \textbf{84} (1986), no.~2, 225--320.

\bibitem{Habegger15}
Philipp Habegger, \emph{Singular moduli that are algebraic units}, Algebra
  Number Theory \textbf{9} (2015), no.~7, 1515--1524. \MR{3404647}

\bibitem{HP17}
Philipp Habegger and Fabien Pazuki, \emph{Bad reduction of genus 2 curves with
  {CM} jacobian varieties}, Compos. Math. \textbf{153} (2017), no.~12,
  2534--2576. \MR{3705297}

\bibitem{Kudla97}
Stephen~S. Kudla, \emph{Central derivatives of {E}isenstein series and height
  pairings}, Ann. of Math. (2) \textbf{146} (1997), no.~3, 545--646.
  \MR{1491448}

\bibitem{KY10}
Stephen~S. Kudla and TongHai Yang, \emph{Eisenstein series for {SL}(2)}, Sci.
  China Math. \textbf{53} (2010), no.~9, 2275--2316. \MR{2718827}

\bibitem{LV15}
Kristin Lauter and Bianca Viray, \emph{On singular moduli for arbitrary
  discriminants}, Int. Math. Res. Not. IMRN (2015), no.~19, 9206--9250.
  \MR{3431591}

\bibitem{Li18}
Yingkun Li, \emph{Average {CM}-values of higher {G}reen's function and
  factorization}, arXiv:1812.08523, 2018.

\bibitem{Schofer}
Jarad Schofer, \emph{Borcherds forms and generalizations of singular moduli},
  J. Reine Angew. Math. \textbf{629} (2009), 1--36. \MR{2527412}

\bibitem{Viazovska11}
Maryna Viazovska, \emph{{CM} values of higher {G}reen's functions},
  arXiv:1110.4654, 2011.

\bibitem{Yang05}
Tonghai Yang, \emph{C{M} number fields and modular forms}, Pure Appl. Math. Q.
  \textbf{1} (2005), no.~2, part 1, 305--340.

\bibitem{YY19}
Tonghai Yang and Hongbo Yin, \emph{Difference of modular functions and their
  {CM} value factorization}, Trans. Amer. Math. Soc. \textbf{371} (2019),
  no.~5, 3451--3482. \MR{3896118}

\bibitem{YYY18}
Tonghai Yang, Hongbo Yin, and Peng Yu, \emph{The lambda invariants at {CM}
  points}, Int. Math. Res. Not. IMRN (to appear).

\bibitem{Zhang97}
Shouwu Zhang, \emph{Heights of {H}eegner cycles and derivatives of
  {$L$}-series}, Invent. Math. \textbf{130} (1997), no.~1, 99--152.
  \MR{1471887}

\end{thebibliography}
\bibliographystyle{amsplain}

\end{document}